\documentclass{amsart}
\usepackage{shortsalch, xy}
\xyoption{all}
\title{Obstructions to compatible splittings.}
\date{February 2014.}

\usepackage[OT2,T1]{fontenc}
\DeclareSymbolFont{cyrletters}{OT2}{wncyr}{m}{n}
\DeclareMathSymbol{\Sha}{\mathalpha}{cyrletters}{"58}
\begin{document}

\begin{abstract}
Suppose one has a map of split short exact sequences in a category of modules,
or more generally, in any abelian category. Do the short exact sequences
split compatibly, i.e., does there exist a splitting of each short exact
sequence which commutes with the map of short exact sequences?
The answer is sometimes yes and sometimes no. We define and prove
basic properties of
an obstruction group to the existence of compatible splittings.
\end{abstract}

\maketitle
\tableofcontents

\section{Introduction.}

Suppose we have a map of short exact sequences (of modules over a ring,
sheaves over a scheme, in general of objects in any abelian category):
\begin{equation}\label{diag 100}\xymatrix{ 
0 \ar[r]\ar[d] & M^{\prime} \ar[r]\ar[d]^g & M\ar[r]\ar[d] & M^{\prime\prime}\ar[r]\ar[d]^f & 0 \ar[d] \\
0 \ar[r] & N^{\prime} \ar[r] & N\ar[r] & N^{\prime\prime} \ar[r] & 0. }\end{equation}
Suppose the top row splits and suppose the bottom row also splits. Does there
exist a {\em compatible} splitting of the two short exact sequences?
In other words, can we find choices of splitting maps of the two rows
which commute with the vertical maps in the diagram?
We call this the {\em compatible splitting problem.}

At a glance one might hope that the answer to this question is always ``yes,''
but in fact the answer is often ``no,'' even for very simple short exact sequences
in very simple (in fact, semisimple!) abelian categories. For example,
the reader can easily verify that, if $k$ is a field, 
\[\xymatrix{ 0 \ar[r]\ar[d] & 0 \ar[r]\ar[d] & k \ar[r]^{\id}\ar[d]^{\id} & k \ar[r]\ar[d] & 0 \ar[d] \\
0 \ar[r] & k \ar[r]^{\id} & k \ar[r] & 0 \ar[r] & 0 }\]
is a map of short exact sequences of $k$-vector spaces in which both the
top and the bottom row split, but the rows cannot be {\em compatibly} split.

Now one asks the natural question: 
\begin{question}\label{the question}
What conditions on $f$ or on $g$ guarantee the existence
of a compatible splitting?\end{question}
This question is an obvious one, 
and one which arose for us during a concrete and practical computation
(see below, at the end of this introduction), but
we could not find any answer to it in the existing literature.
In this paper, we answer this question
by developing an obstruction group to the existence
of certain splittings,
with our main application being to the compatible splitting problem.
As one might expect, our obstruction group 
is the kernel of a map of $\Ext$-groups.
Specifically, if diagram~\ref{diag 100} is a commutative diagram
with split exact rows in an abelian category $\mathcal{C}$, then 
the whole diagram represents some element $\alpha$ of
$\Ext^1_{\mathcal{C}^{(\bullet\rightarrow\bullet)}}(M^{\prime\prime}\stackrel{f}{\rightarrow} 
N^{\prime\prime}, M^{\prime}\stackrel{g}{\rightarrow} N^{\prime})$, where $\mathcal{C}^{(\bullet\rightarrow\bullet)}$ is the (abelian) category of morphisms in $\mathcal{C}$, i.e., functors
from the category $(\bullet\rightarrow \bullet)$, the category with two objects and one morphism between them, to $\mathcal{C}$.
Furthermore, $\alpha$ is in the kernel of the natural map
\begin{equation}\label{diag 1001} \Ext^1_{\mathcal{C}^{(\bullet\rightarrow\bullet)}}(M^{\prime\prime}\stackrel{f}{\rightarrow}
N^{\prime\prime}, M^{\prime}\stackrel{g}{\rightarrow} N^{\prime})
\rightarrow \Ext^1_{\mathcal{C}}(M^{\prime\prime},M^{\prime})\times \Ext^1_{\mathcal{C}}(N^{\prime\prime},N^{\prime})\end{equation}
since the top and bottom rows of diagram~\ref{diag 100} are assumed to split.

The kernel of map~\ref{diag 1001} is our obstruction group
to compatible splitting.
It bears a resemblance to the Tate-Shafarevich group classifying
failure of the Hasse principle for an abelian variety, as in 
e.g. \cite{MR2261462}. For that reason, we adopt the notation
$\Sha^1(f,g)$
for this obstruction group.

We actually define (in Definition~\ref{def of sha}) a sequence of $\Sha^n$ groups and in somewhat wider generality,
so that they apply to not only to the compatible splitting problem as defined above, but also
to longer exact sequences, i.e., those classified by $\Ext^n$ rather than
$\Ext^1$; and so that they are also the home of the splitting obstructions arising
in the following more general context.
One might have an exact sequence in an abelian category, and one might 
know that, on forgetting some structure possessed by objects in that abelian category, the exact sequence is split. One wants to know if the original exact sequence is also split. The compatible splitting problem, described above, is
the special case in which one has a short exact sequence in the category of 
arrows in $\mathcal{C}$, and one forgets the arrows, only remembering
their domains and codomains.
But one could instead, for example, have an exact sequence of
$k[x,y]$-modules which is split as an exact sequence of $k[x]$-modules
and as an exact sequence of $k[y]$-modules,
or one could have, for another example, an exact sequence
of representations of a profinite group $G$ which is split on restricting to
any maximal rank elementary abelian $p$-subgroup of $G$.
(In the latter case, when $G$ has $p$-rank one and finitely many
conjugacy classes of rank $1$ elementary abelian $p$-subgroups,
the group $\Sha^n$ occurs as the kernel of Quillen's map
\[ H^n_c(G; \mathbb{F}_p) \rightarrow \lim_{E\in \mathcal{A}_p(G)} H^n_c(E; \mathbb{F}_p) \]
in continuous group cohomology, from \cite{MR0298694},
where $\mathcal{A}_p$ is 
the category of elementary abelian $p$-subgroups of $G$.)
Our obstruction groups $\Sha^n$, and some of the theorems we prove
about them, are general enough to apply to these situations as well.

Here is a brief synopsis of what we accomplish in each section of the paper.
\begin{itemize}
\item 
In section 2, 
we construct the compatible splitting obstruction groups
$\Sha^n$, and we give a precise formulation of the splitting problems to whose solvability these
groups are the obstructions. 
\item 
Like the classical Tate-Shafarevich group, and as is clear
from the case when the $\Sha^n$ groups are the kernel of Quillen's
map above, the groups $\Sha^n$ are difficult to compute.
This difficulty arises in part because they are not (co)homological, that is,
they do not turn short exact sequences (other than split ones) 
into exact sequences.
So in section 3, 
we produce a relative-cohomological approximation $\Ext_E^n$ to
$\Sha^n$, which has the advantage that it 
turns (certain) non-split exact sequences into
long exact sequences.
In Theorem~\ref{sha agrees with relative ext} 
we prove a ``Hurewicz theorem'' for this cohomological approximation,
i.e., we produce a natural transformation between this $\Ext_E^n$
group and $\Sha^n$ which is an isomorphism when $n=1$.
\item In section 4, 
we produce the ``compatible splitting spectral sequence,'' which relates the higher $\Ext^n_E$ groups and the higher 
$\Sha^n$ groups.
As a corollary, we show that,
when the allowable class $E$ defining these $\Ext^n_E$-groups is hereditary,
the groups $\Sha^n$ and $\Sha^{n+1}$ fit into a certain exact sequence. 
This relationship between $\Sha^n$ and $\Sha^{n+1}$ is a curious duality-like phenomenon which in fact occurs
(the relative-hereditary condition is satisfied) in 
our most important application, the compatible splitting problem, as we demonstrate in the next section.
\item Finally, in section 5 
we specialize to the case of the compatible splitting problem.
In Corollary~\ref{split monics cor}, we use our cohomological approximation to $\Sha$ to 
show that, for a fixed choice of map $g$,
a compatible splitting exists for all diagrams of the 
form~\ref{diag 100} with split exact rows if and only if $g$
is split epic.
Dually, in Corollary~\ref{split epics cor}, we show that for a fixed choice of map $f$,
a compatible splitting exists for all diagrams of the 
form~\ref{diag 100} with split exact rows if and only if $f$
is split monic. (These last two sentences together constitute the simplest and straightforward answer, but not the most
general answer, to our Question~\ref{the question}.)
In Corollary~\ref{splitting duality}, 
we then demonstrate {\em compatible splitting duality},
a concrete special case of the duality described in the previous section:
the obstruction group $\Sha^n(f,g)$ is, by definition,
the {\em kernel} of the map
\begin{equation}\label{diag 1002} \Ext^n_{\mathcal{C}^{\bullet\rightarrow\bullet}}(f,g)
\rightarrow \Ext^n_{\mathcal{C}}(\dom f, \dom g)\times \Ext^n_{\mathcal{C}}(\cod f,\cod g),\end{equation}
but ``compatible splitting duality'' identifies 
the {\em cokernel} of the natural map
\[ \Ext^n_{\mathcal{C}}(\dom f, \dom g)\times \Ext^n_{\mathcal{C}}(\cod f,\cod g)
\rightarrow \Ext^n_{\mathcal{C}}(\dom f, \cod g)\]
with the compatible splitting obstruction group in the {\em next dimension,}
that is, $\Sha^{n+1}(f,g)$. This gives a curious relation
between the obstruction groups in adjacent dimensions which we think
is rather surprising.
\end{itemize}

We use these results in our paper~\cite{computationofmodelstructures}
in the course of a classification of all morphisms between modules over
length $2$ local principal Artin rings, such as $k[x]/x^2$ or $\mathbb{Z}/p^2\mathbb{Z}$. Our use of these results in that paper
boils down to the following issue: suppose 
$h: M\rightarrow N$ is a map of $R$-modules, and suppose $M_0$ is a direct summand
of $M$. Suppose we can even show that the image of $M_0$ under $h$ is a 
direct summand of $N$. Can we then split $h$ as a direct sum of the 
map $h\mid_{M_0}: M_0\rightarrow N_0$ with another map? 
This is precisely the compatible splitting problem:
the map $h\mid_{M_0}$ is the map $g$ as in diagram~\ref{diag 100},
and the quotient map $h/(h\mid_{M_0})$ is the map $f$
as in diagram~\ref{diag 100}. So we have found that the results in this paper 
are quite useful even for some very elementary algebraic tasks.

\section{The compatible splitting obstruction group, and its cohomological approximation.}

\begin{definition}\label{def of sha}
Let $\mathcal{C}$ be an abelian category, $I$ a finite set, and $\{\mathcal{C}_i: i\in I\}$ a finite set of abelian categories.
Suppose that, for each $i\in I$, we have a faithful exact functor $G_i: \mathcal{C}\rightarrow \mathcal{C}_i$ and a left adjoint
$F_i$ for $G_i$. 
We will write $G$ for the resulting functor $G: \times_{i\in I} \mathcal{C}_i \rightarrow \mathcal{C}$
given by $G((X_i)_{i\in I}) = \oplus_{i\in I} X_i$, and $F$ for its left adjoint
$F: \mathcal{C}\rightarrow \times_{i\in I}\mathcal{C}_i$ given by letting its component in the factor $\mathcal{C}_i$ be $F_i$.
(It is easy to show that $F$ is indeed left adjoint to $G$.)
Then, for any $n\in \mathbb{N}$ and any objects $X,Y$ of $\mathcal{C}$, by the {\em $n$th compatible splitting obstruction group $\Sha^n(X, Y)$} we shall mean the kernel of the map
\[ \Ext^n_{\mathcal{C}}(X, Y) \rightarrow \Ext^n_{\mathcal{C}}(FG X , Y)\]
induced by the counit map $\epsilon_X: FGX \rightarrow X$.
\end{definition}
\begin{remark}
We use the symbol $\Sha$ to denote the compatible splitting obstruction groups because of their similarity, both in its definition and its properties, 
to the higher Shafarevich-Tate groups of an abelian variety (see e.g. \cite{MR2261462}).

Note that $\Sha^n(X, Y)$ certainly depends on the choices made for $\{\mathcal{C}_i: i\in I\}$ and $\{ F_i: i\in I\}$, but 
in order to keep the notation manageable, we suppress these choices from the notation for $\Sha^n(X, Y)$.

Note also that we do {\em not} need to assume that $\mathcal{C}$ has enough injectives or enough projectives for Definition~\ref{def of sha} to make sense:
$\Ext^n_{\mathcal{C}}$ defined after Yoneda, as equivalence classes of length $n+2$ exact sequences in $\mathcal{C}$, does not require
injective or projective resolutions. See e.g. \cite{MR1344215} for basic material on $\Ext$ without enough injectives or enough projectives.
\end{remark}

\begin{definition} Let $\mathcal{C},\mathcal{D}$ be abelian categories, and let $F: \mathcal{D}\rightarrow \mathcal{C}$ be a additive functor.
We say that $F$ is {\em resolving} if, for every object $X$ of $\mathcal{C}$, there exists a projective object $Y$ of $\mathcal{C}$ such that
$F(Y)$ is a projective object in $\mathcal{D}$ and such that there exists an epimorphism $Y \rightarrow X$ in $\mathcal{C}.$
\end{definition}
For example, if $\mathcal{C},\mathcal{D}$ are categories of modules over rings, and $F$ sends free modules to free modules (e.g. $F$ could be a base-change/extension
of scalars functor), then $F$ is resolving.
\begin{lemma}\label{change of rings lemma}
Let $\mathcal{C},\mathcal{D}$ be abelian categories, $G: \mathcal{C}\rightarrow\mathcal{D}$ a functor with exact left adjoint $F$.
Suppose $F$ is resolving. Then, for all $n\in\mathbb{N}$, we have an isomorphism
\[ \Ext^n_{\mathcal{C}}(FGX, Y) \cong \Ext^n_{\mathcal{D}}(GX, GY)\]
natural in $X$ and $Y$.
\end{lemma}
\begin{proof}
Since $F$ is resolving, we can choose, for every object $X$ of $\mathcal{C}$, a chain complex $P_{\bullet}$ in $\mathcal{D}$
such that $P_n$ is projective in $\mathcal{D}$ for all $n$, such that $FP_n$ is projective in $\mathcal{C}$ for all $n$, and such that
the homology of $P_{\bullet}$ is $GX$, concentrated in degree zero. Since $F$ is exact, $FP_{\bullet}$ is a projective resolution of $FGX$.
So now we have (natural!) isomorphisms
\begin{eqnarray*} \Ext^n_{\mathcal{C}}(FGX, Y) & \cong & H^n(\hom_{\mathcal{C}}(FP_{\bullet}, Y)) \\
 & \cong & H^n(\hom_{\mathcal{D}}(P_{\bullet}, GY)) \\
 & \cong & \Ext^n_{\mathcal{D}}(GX, GY).\end{eqnarray*}
\end{proof}
Exactness of $F$ is necessary in Lemma~\ref{change of rings lemma}. For example, suppose $\mathcal{C} = \Mod(k)$ and $\mathcal{D} = \Mod(k[x])$,  and 
suppose that $F$ and $G$
are the induction and restriction of scalars functors, respectively, induced by the ring map $k[x]\rightarrow k$ sending $x$ to $0$. Then
one sees easily that $F$ is resolving but not exact, and that the conclusion of Lemma~\ref{change of rings lemma} fails dramatically.

\begin{definition}\label{def of allowable class}
Let $\mathcal{C},I,\{ \mathcal{C}_i: i\in I\}$ be as in Definition~\ref{def of sha}.
Suppose $F_i$ is exact and resolving for all $i\in I$. We let $E$ denote
the following allowable class in $\mathcal{C}$, in the sense of relative homological algebra:
$E$ consists of the class of all short exact sequences in $\mathcal{C}$ of the form
\[ 0 \rightarrow \ker\epsilon_Y \rightarrow  FG Y \stackrel{\epsilon_Y}{\longrightarrow} Y \rightarrow 0,\]
where $Y$ ranges across all objects of $\mathcal{C}$. (The map $\epsilon_Y$ is an epimorphism since each $G_i$ is assumed faithful,
hence each $F_i G_i Y\rightarrow Y$ is individually already an epimorphism, by e.g. theorem 1 of section IV.3 of \cite{MR1712872}.)
\end{definition}
Clearly $E$ depends on the choices made for $\{\mathcal{C}_i: i\in I\}$ and $\{ F_i: i\in I\}$, but 
in order to keep the notation manageable, we suppress these choices from the notation for $E$.

The following (easy!) theorem makes clear why the compatible splitting obstruction group is useful.
\begin{theorem}\label{main general thm}
Let $\mathcal{C},I,\{ \mathcal{C}_i: i\in I\}$ be as in Definition~\ref{def of sha}.
Suppose $F_i$ is exact and resolving for all $i\in I$.
Let $X,Y$ be objects of $\mathcal{C}$. Then the following conditions are equivalent:
\begin{itemize}
\item $\Sha^n(X, Y) \cong 0$.
\item For every length $n+2$ exact sequence $\alpha$ of the form
\[ 0 \rightarrow Y \rightarrow E_1 \rightarrow \dots \rightarrow E_n \rightarrow X\rightarrow 0,\]
$\alpha$ is split (in $\mathcal{C}$) if and only if, for all $i\in I$, the exact sequence $G_i(\alpha)$ is split (in $\mathcal{C}_i$).
\end{itemize}

If $n=1$, the above conditions are furthermore each equivalent to:
\begin{itemize}
\item For every morphism $f: \ker \epsilon_Y \rightarrow X$ in $\mathcal{C}$, there exists a map $g$ making the diagram
\[\xymatrix{
\ker \epsilon_Y \ar[r] \ar[d]_f & FGY \ar[ld]^g \\
X & }\] commute.
\end{itemize}
\end{theorem}
\begin{proof}
From Lemma~\ref{change of rings lemma} we have a commutative diagram
\[ \xymatrix{ \Ext^n_{\mathcal{C}}(Y, X) \ar[r]\ar[rd] & \Ext^n_{\mathcal{C}}( F G Y, X) \ar[d]^{\cong} \\
 & \Ext^n_{\mathcal{C}}( \oplus_{i\in I} F_i G_i Y, X) \ar[d]^{\cong} \\
 & \oplus_{i\in I} \Ext^n_{\mathcal{C}_i}(G_i Y, G_i X).}\]
Now an element $\alpha\in \Ext^n_{\mathcal{C}}(Y, X)$ maps to zero in $\oplus_{i\in I} \Ext^n_{\mathcal{C}_i}(G_i Y, G_i X)$ if and only if
$G_i(\alpha) = 0\in \Ext^n_{\mathcal{C}_i}(G_i Y, G_i X)$ for each $i\in I$, i.e., if and only if
the length $n+2$ exact sequence represented by $\alpha$ becomes split in $\mathcal{C}_i$ after applying $G_i$,
for all $i\in I$.

The claim for $n=1$ follows from the exact sequence
\[ \hom_{\mathcal{C}}(FG Y, X) \rightarrow \hom_{\mathcal{C}}(\ker \epsilon_Y, X) \rightarrow \Ext_{\mathcal{C}}^1(Y, X) \rightarrow \Ext_{\mathcal{C}}^1(FG Y, X),\]
hence the exact sequence
\[ \hom_{\mathcal{C}}(FG Y, X) \rightarrow \hom_{\mathcal{C}}(\ker \epsilon_Y, X) \rightarrow \Sha^1(Y, X) \rightarrow 0.\]
\end{proof}

\begin{corollary}
Let $\mathcal{C},I,\{ \mathcal{C}_i: i\in I\}$ be as in Definition~\ref{def of sha}, and let $E$ be as in 
Definition~\ref{def of allowable class}.
Then the following are equivalent, for a given object $X$ of $\mathcal{C}$:
\begin{itemize}
\item $\Sha^1(Y, X) \cong 0$ for all objects $Y$ of $\mathcal{C}$.
\item For every short exact sequence
\[ \alpha = \left( 0 \rightarrow X \rightarrow Y \rightarrow Z \rightarrow 0 \right)\]
in $\mathcal{C}$, the short exact sequence $\alpha$ splits (in $\mathcal{C}$) if and only if $G_i(\alpha)$ splits (in $\mathcal{C}_i$) for all $i$.
\item $X$ 
is $E$-injective.
\end{itemize}
\end{corollary}
\begin{proof}
That the first and second conditions are equivalent follows from Theorem~\ref{main general thm}. 
The third condition is plainly the general (for all objects $Y$) form of the third condition from Theorem~\ref{main general thm}, hence
equivalent to the first two.
\end{proof}

Another corollary of Theorem~\ref{main general thm} is provided in Corollary~\ref{diagram splittings corollary}.

\section{The Hurewicz theorem.}

One knows the Hurewicz theorem from classical algebraic topology: there exists a natural transformation $\pi_*\rightarrow H_*$, that is, from 
the homotopy groups functor 
to the homology groups functor; in degree $1$ it is the abelianization functor; and if $\pi_i(X)$ vanishes for all $1\leq i\leq n-1$, then the degree $n$
Hurewicz map $\pi_n(X)\rightarrow H_n(X)$ is an isomorphism. Hence, while homotopy is difficult to compute but of great intrinsic interest, 
homology is homological (turns cofiber sequences to long exact sequences) and hence much easier to compute, and the Hurewicz theorem tells us that
homology is a good homological approximation to homotopy. 

We now prove a Hurewicz theorem which compares the $\Sha$ groups to the relative $\Ext$-groups $\Ext_{\mathcal{C}/E}$. The moral to this theorem is the following:
the $\Sha$ groups are not (co)homological, they do not turn (non-split) short exact sequences to long exact sequences, making them difficult to compute. 
However, 
the $\Ext_{\mathcal{C}/E}$-groups turn $E$-short exact sequences to long exact sequences, making them (in principle, and sometimes also in practice)
easier to compute than the $\Sha$ groups: see Theorem~\ref{main thm on morphisms}, for example, where we prove a broad vanishing theorem for
the $\Ext_{\mathcal{C}/E}$-groups under circumstances where the $\Sha$-groups do not usually vanish and are usually quite nontrivial to compute.
As $\Ext_{\mathcal{C}/E}$ is cohomological rather than homological (and $\Sha$ is cohomotopical rather than homotopical?), our Hurewicz natural transformation
goes from $\Ext_{\mathcal{C}/E}^n$ to $\Sha^n$, rather than the reverse. This natural transformation is an isomorphism in degree $1$, just as in the classical
Hurewicz theorem, as we now prove:
\begin{theorem}\label{sha agrees with relative ext} {\bf (Hurewicz theorem for $\Sha$.)} 
Let $\mathcal{C},I,\{ \mathcal{C}_i: i\in I\}$ be as in Definition~\ref{def of sha}, and let $E$ 
be the allowable class defined in Definition~\ref{def of allowable class}.
Suppose $F_i$ is exact and resolving for all $i\in I$.
Then the category $\mathcal{C}$ has enough
$E$-projectives. Furthermore, for each object $X$ of $\mathcal{C}$ and each $n\in \mathbb{N}$, there exists a natural transformation
\[ H: \Ext^n_{{\mathcal{C}/E}}(-, X) \rightarrow \Sha^n(-, X),\]
which we call the {\em Hurewicz map for $\Sha$.} 
In degree one, the Hurewicz map is a natural isomorphism:
\[ \Ext^1_{{\mathcal{C}/E}}(-, X) \stackrel{\cong}{\longrightarrow} \Sha^1(-, X).\]
\end{theorem}
\begin{proof}
\begin{itemize}
\item {\em Existence of enough $E$-projectives:}
We claim that, for every object $X$ of $\mathcal{C}$, the object $FGX$ is $E$-projective.
Indeed, every $E$-epimorphism is of the form $FGY\stackrel{\epsilon_Y}{\longrightarrow} Y$ for some object $Y$ of $\mathcal{C}$,
so suppose we have a diagram
\begin{equation}\label{diag 6} \xymatrix{ & FGX \ar@{-->}[ld]_g \ar[d]^f \\ FGY\ar[r]_{\epsilon_Y} & Y }\end{equation}
and want to produce a morphism as in the dotted arrow making the diagram commute, so that $FGX$ satisfies the universal property for an $E$-projective object.
Using the adjunction $F\dashv G$, diagram~\ref{diag 6} is equivalent to the diagram
\[ \xymatrix{ & GX \ar@{-->}[ld]_h \ar[d]^{f^{\flat}} \\ GY\ar[r]_{\id_Y} & GY }\]
in the product category $\times_{i\in I} \mathcal{C}_i$,
and now the map $h$ exists: it is simply $h=f^{\flat}$.
Now the desired map $g$ in diagram~\ref{diag 6} is simply $g = Fh = Ff^{\flat}$.
Hence $FGX$ is $E$-projective.
Hence the short exact sequence
\[ 0 \rightarrow \ker\epsilon_X \rightarrow FGX\stackrel{\epsilon_X}{\longrightarrow} X \rightarrow 0\]
shows that there exists an $E$-epimorphism from an $E$-projective to $X$. So $\mathcal{C}$ has enough $E$-projectives.
\item {\em The Hurewicz map in degree $1$:}
Now, for any objects $X,Y$ in $\mathcal{C}$, we have the exact sequence
\[ \hom_{\mathcal{C}}(FGY, X) \rightarrow \hom_{\mathcal{C}}(\ker \epsilon_Y, X) \rightarrow \Ext^1_{\mathcal{C}/E}(Y, X)
 \rightarrow \Ext^1_{\mathcal{C}/E}(FGY, X) \]
and $\Ext^1_{\mathcal{C}}(FGY, X)\cong 0$ since $FGY$ is $E$-projective. So $\Ext^1_{\mathcal{C}/E}(Y, X)$
is the cokernel of the map $\hom_{\mathcal{C}}(FGY, X) \rightarrow \hom_{\mathcal{C}}(\ker\epsilon_Y, X)$.
Meanwhile, we have the exact sequence
\[ \hom_{\mathcal{C}}(FGY, X) \rightarrow \hom_{\mathcal{C}}(\ker\epsilon_Y, X) \rightarrow \Ext^1_{\mathcal{C}}(Y, X)
 \rightarrow \Ext^1_{\mathcal{C}}(FGY, X), \]
in which the cokernel of the left-hand map is $\Ext^1_{\mathcal{C}/E}(Y, X)$, and the kernel
of the right-hand map is, by definition, $\Sha^1(Y, X)$. Hence the natural isomorphism $\Ext^1_{\mathcal{C}/E}(Y, X)\cong \Sha^1(Y, X)$.
\item {\em Construction of the Hurewicz map in degrees $>1$:}
Suppose $n>1$.
Since the composite map
\[ \Ext_{\mathcal{C}}^{n-1}(\ker \epsilon_Y, X)\rightarrow \Ext_{\mathcal{C}}^{n}(Y, X) \rightarrow \Ext_{\mathcal{C}}^{n}(FGY, X)\]
is zero, the map $\Ext_{\mathcal{C}}^{n-1}(\ker \epsilon_Y, X)\rightarrow \Ext_{\mathcal{C}}^{n}(Y, X)$ factors through the inclusion of the
kernel $\Sha^{n}(Y, X)\hookrightarrow \Ext_{\mathcal{C}}^n(Y, X)$ of the right-hand map.
So we have a factor map $f: \Ext_{\mathcal{C}}^{n-1}(\ker \epsilon_Y, X) \rightarrow \Sha^n(Y, X)$.
Now the Hurewicz map, when applied to an object $Y$, is simply the composite
\[ \Ext^n_{\mathcal{C}/E}(Y, X) \stackrel{\cong}{\longrightarrow} \Ext^n_{\mathcal{C}/E}(\ker \epsilon_Y, X) \stackrel{f}{\longrightarrow}
\Sha^n(Y, X) .\]
\end{itemize}
\end{proof}
\begin{remark}
One would like to know, by analogy with the classical Hurewicz theorem in topology, whether vanishing of $\Sha^i(Y, X)$ and $\Ext^i_{\mathcal{C}/E}(Y, X)$
for all $i<n$ implies that the degree $n$ Hurewicz map $\Ext^n_{\mathcal{C}/E}(Y, X)\rightarrow \Sha^n(Y, X)$ is an isomorphism.
Clearly, this is true if one assumes that $\Sha^i(Y, X)$ and $\Ext^i_{\mathcal{C}/E}(Y, X)$ vanish for $i<n$ {\em for all $Y$}, since as long as $n\geq 2$ this implies that
$X$ is $E$-projective and hence that $\Sha^i(Y, X)\cong \Ext^i_{\mathcal{C}/E}(Y, X)\cong 0$ for all $i$ and all $Y$. But it is natural to ask instead if the degree $n$
Hurewicz map is an isomorphism if all lower $\Ext_{\mathcal{C}/E}$ vanish, ``one object at a time.'' We do not know the answer to this question.
\end{remark}

\section{The spectral sequence.}

Here is a natural spectral sequence which is {\em not} the one we will use in this paper! We describe it because its construction is slightly more obvious than
the one we {\em will} use, and we want to avoid the reader mistaking one spectral sequence for the other.
\begin{prop}{\bf (The change-of-allowable-class spectral sequence.)}
Let $\mathcal{C}$ be an abelian category, and let $D,E$ be allowable classes in $\mathcal{C}$.
Suppose $D\subseteq E$, suppose that $\mathcal{C}$ has enough $D$-projectives and enough $E$-projectives.
Suppose $X,Y$ are objects of $\mathcal{C}$, and choose an $E$-projective $E$-resolution
\[ \dots \rightarrow P_2 \rightarrow P_1 \rightarrow P_0 \rightarrow X \rightarrow 0\]
of $Y$.
Then there exists a spectral sequence
\begin{eqnarray*} E_1^{s,t} \cong \Ext_{\mathcal{C}/D}^t(P_s, X) & \Rightarrow & \Ext_{\mathcal{C}/E}^{s+t}(Y, X)\\
 d_r: E_r^{s,t} & \rightarrow & E_r^{s+r, t-r+1}.\end{eqnarray*}
\end{prop}
\begin{proof}
Special case of the usual resolution spectral sequence, as in e.g. Thm A1.3.2 of~\cite{MR860042}, arising from applying
$\Ext_{\mathcal{C}/D}$ to 
\[ \dots \rightarrow P_2 \rightarrow P_1 \rightarrow P_0 \rightarrow  0.\]
\end{proof}

By contrast, the following spectral sequence is the one more relevant to the compatible splitting obstruction groups.
We will write ``absolute projectives'' to mean the ordinary, usual projective objects in a category, because we shall
need to refer to both relative projectives, that is, $E$-projectives, and the absolute projectives, and we want our terminology to be
as unambiguous as possible.
\begin{theorem}\label{main ss}{\bf (The compatible splitting spectral sequence.)}
Let $\mathcal{C},I,\{ \mathcal{C}_i: i\in I\}$ be as in Definition~\ref{def of sha}, and let $E$ 
be the allowable class defined in Definition~\ref{def of allowable class}.
Suppose $F_i$ is exact and resolving for all $i\in I$, and suppose $\mathcal{C}$ has enough absolute projectives.
Let $X,Y$ be objects in $\mathcal{C}$, and define sequences of objects $U_i,V_i$ in $\mathcal{C}$ inductively as follows:
let $U_0 = V_0 = Y$, and for all $i\geq 1$, let $U_i = FGV_i$ and let $V_i = \ker (\epsilon_{V_{i-1}}: U_{i-1}\rightarrow V_{i-1})$.
Then there exists a spectral sequence
\begin{eqnarray*} E_1^{s,t} \cong \Ext_{\mathcal{C}}^t(U_s, X) & \Rightarrow & 0 \\
 d_r: E_r^{s,t} & \rightarrow & E_r^{s+r, t-r+1}\end{eqnarray*}
with the following properties:
\begin{itemize}
\item As stated above, this spectral sequence converges to the zero bigraded abelian group.
\item We have an identification of the $E_2$-page of the spectral sequence:
\[ E_2^{s,t} \cong \left\{ \begin{array}{lll} \left(R^{s-1}_E\Ext_{\mathcal{C}}^t(-,X)\right)(Y) & \mbox{\ if\ } & s\geq 2 \\
\left(\left(R^{0}_E\Ext_{\mathcal{C}}^t(-,X)\right)(Y)\right)/\left(\Ext_{\mathcal{C}}^t(Y,X)\right) & \mbox{\ if\ } & s=1 \\
\Sha^t(Y,X) & \mbox{\ if\ } & s=0. \end{array} \right. \]
\item In particular, $E_2^{s,0}\cong \Ext_{\mathcal{C}/E}^{s-1}(Y, X)$ if $s\geq 2$, and $E_2^{0,0}\cong E_2^{1,0}\cong 0$.\end{itemize}
\end{theorem}
\begin{proof}
Special case of the usual resolution spectral sequence, as in e.g. Thm A1.3.2 of~\cite{MR860042}, arising from applying
$\Ext_{\mathcal{C}}$ to 
the long exact sequence
\begin{equation}\label{resolution 1} \dots \rightarrow U_3 \rightarrow U_2 \rightarrow U_1 \rightarrow U_0\rightarrow 0\end{equation}
obtained by splicing the short exact sequences
\[ 0 \rightarrow V_{i+1}\rightarrow U_i\rightarrow  V_i\rightarrow 0.\]
In more detail: long exact sequence~\ref{resolution 1} is an $E$-projective $E$-resolution for $U_0 = Y$.
We choose an absolute projective resolution for each $U_i$ and obtain a double complex:
\[\xymatrix{
  & \vdots \ar[d] & \vdots \ar[d] & \vdots \ar[d] & \vdots \ar[d] \\
\dots\ar[r] & P_{2,2}\ar[r]\ar[d] & P_{2,1} \ar[r]\ar[d] & P_{2,0}\ar[r]\ar[d] & 0 \ar[d] \\
\dots\ar[r] & P_{1,2}\ar[r]\ar[d] & P_{1,1} \ar[r]\ar[d] & P_{1,0}\ar[r]\ar[d] & 0 \ar[d] \\
\dots\ar[r] & P_{0,2}\ar[r]\ar[d] & P_{0,1} \ar[r]\ar[d] & P_{0,0}\ar[r]\ar[d] & 0 \ar[d] \\
\dots\ar[r] & 0 \ar[r] & 0 \ar[r] & 0 \ar[r] & 0 }\]
such that each $P_{i,j}$ is an absolute 
projective in $\mathcal{C}$, the rows are exact, the homology of the column $P_{\bullet, i}$ is $U_i$ concentrated in degree $0$,
and the maps induced in degree $0$ homology by the horizontal differentials in the double complex are the maps in resolution~\ref{resolution 1}.
Then we apply $\hom_{\mathcal{C}}(-, X)$ to the entire double complex to yield a new double complex $\hom_{\mathcal{C}}(P_{\bullet,\bullet}, Y)$.
Now we have the two spectral sequences of the double complex $\hom_{\mathcal{C}}(P_{\bullet,\bullet}, Y)$, as in \cite{MR1731415}:
in the first spectral sequence, the $E_1$-term is given by the cohomology of the rows in $\hom_{\mathcal{C}}(P_{\bullet,\bullet}, Y)$,
each of which is $\Ext_{\mathcal{C}}^*(0, Y)$, since each row is $\hom_{\mathcal{C}}(-, Y)$ applied to a projective resolution of the zero object.
Hence the spectral sequence is zero in the $E_1$-term, hence zero in the $E_\infty$-term. The two spectral sequences converge to the same object,
hence the second spectral sequence converges to zero.
In the second spectral sequence, the $E_1$-term is given by the cohomology 
of the columns in $\hom_{\mathcal{C}}(P_{\bullet,\bullet}, Y)$, i.e., $\Ext_{\mathcal{C}}^*(U_*, Y)$. 

To prove our claims about $E_2$ we need to examine the $d_1$ differential in this second spectral sequence.
Along the rows of the spectral sequence, the $d_1$ differential is the differential of a cochain complex
\[ \dots \stackrel{d_1}{\longrightarrow}\Ext^t_{\mathcal{C}}(U_s, X) \stackrel{d_1}{\longrightarrow} \Ext^t_{\mathcal{C}}(U_{s+1}, X) \stackrel{d_1}{\longrightarrow} \Ext^t_{\mathcal{C}}(U_{s+2}, X)\stackrel{d_1}{\longrightarrow} \dots , \]
and since~\ref{resolution 1} is an $E$-projective $E$-resolution of $Y$, if $s\geq 2$, the cohomology of this cochain complex is
$(R^{s-1}_E \Ext^t_{\mathcal{C}}(- ,X))(Y)$, the $(s-1)$th $E$-right-derived functor of $\Ext^t_{\mathcal{C}}(-, X)$.
The degree shift as well as the $s=0$ and $s=1$ special cases are because we did not truncate the degree $0$ part of
resolution~\ref{resolution 1} before applying $\hom_{\mathcal{C}}(-, X)$ as one typically does when computing a derived functor;
instead we left $Y$ in its place in the long exact sequence when we applied $\hom_{\mathcal{C}}(-, X)$.
Since $(R^0_E\Ext_{\mathcal{C}}^t(-,X))(Y)$ is the kernel of the map
\[ \Ext^t_{\mathcal{C}}(U_1, X) \stackrel{d_1}{\longrightarrow} \Ext^t_{\mathcal{C}}(U_{2}, X),\]
we have that $E_2^{0,t}$ is isomorphic to the cokernel of the map
\[ \Ext_{\mathcal{C}}^t(Y,X) \rightarrow (R^0_E\Ext_{\mathcal{C}}^t(-,X))(Y),\]
as claimed in the statement of the theorem.
\end{proof}
Note that the transgression map on the $E_{n+1}$-page of the compatible splitting spectral sequence goes from $\Sha^n(Y,X)$ to $\Ext^n_{\mathcal{C}/E}(Y, X)$.
Curiously, this is the reverse direction of the Hurewicz map of Theorem~\ref{sha agrees with relative ext}. The transgression, however,
does {\em not} yield a natural transformation $\Sha^n(-,X) \rightarrow \Ext^n_{\mathcal{C}/E}(-, X)$,
since for $n>1$, $\Sha^n(Y, X)$ might support a differential before the $E_{n+1}$-page, and
$\Ext^n_{\mathcal{C}/E}(-, X)$ might be hit by a differential before the $E_{n+1}$-page.
For $n=1$ note that the above transgression is a $d_2$-differential which must be an isomorphism in order for the spectral sequence
to converge to zero. This gives another proof that $\Sha^1$ agrees with $\Ext^1_{\mathcal{C}/E}$, as in Theorem~\ref{sha agrees with relative ext}.

The $s=0$ line in the $E_2$-term of the compatible splitting spectral sequence measures the failure of the functors $\{ G_i\}_{i\in I}$ to detect
splitting of finite-length exact sequences, in the sense made precise in Theorem~\ref{main general thm}. 
We also have a conceptual interpretation of the $s=0$ and $s=1$ lines of the $E_2$-term, taken together, in the compatible splitting spectral sequence:
these two lines measures the failure of $\Ext_{\mathcal{C}}$ to be {\em left $E$-exact.} More precisely:
\begin{corollary}\label{left exactness and vanishing columns}
Let $\mathcal{C},I,\{ \mathcal{C}_i: i\in I\}$ be as in Definition~\ref{def of sha}, and let $E$ 
be the allowable class defined in Definition~\ref{def of allowable class}.
Suppose $F_i$ is exact and resolving for all $i\in I$, and suppose further that
$\mathcal{C}$ has enough $E$-injectives. Let $X$ be an object of $\mathcal{C}$, and 
let $t$ be a nonnegative integer.
Then the functor $\Ext^t_{\mathcal{C}}(- ,X)$ is left $E$-exact if and only if
the groups $E_2^{0,t}$ and $E_2^{1,t}$ vanish, for all objects $Y$, in the spectral sequence
of Theorem~\ref{main ss}.
\end{corollary}
\begin{proof}
If $\mathcal{C}$ has enough $E$-injectives, then the natural map to the $0$th right satellite 
$\Ext^t_{\mathcal{C}}(-, X)\rightarrow R^0_E\Ext^t_{\mathcal{C}}(-, X)$ is an isomorphism if and only if $\Ext^t_{\mathcal{C}}(-, X)$
is left $E$-exact. However, $(R^0_E\Ext^t_{\mathcal{C}}(-, X))(Y)$ is the kernel of the map
$d_1: \Ext^t_{\mathcal{C}}(U_1,X) \rightarrow \Ext^t_{\mathcal{C}}(U_2,X)$, so the natural map to the $0$th right satellite fits into the exact sequence
\[ 0 \rightarrow E_2^{0,t} \rightarrow \Ext^t_{\mathcal{C}}(Y,X) \rightarrow (R^0_E\Ext^t_{\mathcal{C}}(-, X))(Y) \rightarrow E_2^{1,t} \rightarrow 0.\]
Hence this natural map is an isomorphism if and only if both $E_2^{0,t}$ and $E_2^{1,t}$ vanish.
\end{proof}

\begin{corollary}\label{compatible splitting duality}
Let $\mathcal{C},I,\{ \mathcal{C}_i: i\in I\}$ be as in Definition~\ref{def of sha}, and let $E$ 
be the allowable class defined in Definition~\ref{def of allowable class}.
Suppose $F_i$ is exact and resolving for all $i\in I$, and suppose
that $Y$ is an object such that $\Ext_{\mathcal{C}/E}^n(Y, X)$
is trivial for all $n>1$ and all $X$.
Then, for all $t\geq 1$ and all objects $X,Y$ of $\mathcal{C}$, we have natural isomorphisms
\begin{eqnarray*} \Sha^t(Y, X) & \cong & (R^1_E\Ext^{t-1}_{\mathcal{C}}(-,X))(Y) \\
 & \cong & \Ext^{t-1}_{\mathcal{C}}(FG\ker \epsilon_Y, X)/\Ext^{t-1}_{\mathcal{C}}(FGY, X) .\end{eqnarray*}
Furthermore, if $\mathcal{C}$ has enough $E$-injectives, then for all $t\geq 1$, the functor $\Sha^t(-, X)$ vanishes if and only if
$\Ext^{t-1}_{\mathcal{C}}(-, X)$ is left $E$-exact.
\end{corollary}
\begin{proof}
The assumption implies that $Y$ has an $E$-projective $E$-resolution of length $1$. So $E$-right derived functors 
of contravariant functors on $\mathcal{C}$ vanish above degree $1$ when applied to $Y$. (We are taking projective, not injective, resolutions in order to take right derived functors,
because we are taking derived functors of {\em contravariant} functors.)

Consequently, by the identification of the $E_2$-term in Theorem~\ref{main ss}, 
the compatible splitting spectral sequence is concentrated in the $s=0,s=1,$ and $s=2$ columns.
Since the spectral sequence must converge to the zero bigraded abelian group, this implies that
the $s=1$ line vanishes, and that the $d_2$-differential is an isomorphism. This implies the isomorphisms
\[ \Sha^t(Y, X) \cong (R^1_E\Ext^{t-1}_{\mathcal{C}}(-,X))(Y) \cong  \Ext^{t-1}_{\mathcal{C}}(FG\ker \epsilon_Y, X)/\Ext^{t-1}_{\mathcal{C}}(FGY, X).\]
Since the $s=1$ line vanishes, by Corollary~\ref{left exactness and vanishing columns}
the functor $\Ext^{t-1}_{\mathcal{C}}(-, X)$ is left $E$-exact if and only if $E_2^{0,t-1}$ vanishes for all $Y$, i.e., if and only if
$\Sha^t(Y,X)$ vanishes for all $Y$.
\end{proof}

The following duality corollary is the one we use in our most important application, in Corollary~\ref{splitting duality}.
\begin{corollary}\label{compatible splitting duality in hereditary case}
Let $\mathcal{C},I,\{ \mathcal{C}_i: i\in I\}$ be as in Definition~\ref{def of sha}, and let $E$ 
be the allowable class defined in Definition~\ref{def of allowable class}.
Suppose $F_i$ is exact and resolving for all $i\in I$, and suppose
further that $E$ is {\em hereditary}, that is, the relative $\Ext$-groups $\Ext_{\mathcal{C}/E}^n(-, -)$
are trivial for all $n>1$.
Then, for all $t\geq 1$ and all objects $X,Y$ of $\mathcal{C}$, we have natural isomorphisms
\begin{eqnarray*} \Sha^t(Y, X) & \cong & (R^1_E\Ext^{t-1}_{\mathcal{C}}(-,X))(Y) \\
 & \cong & \Ext^{t-1}_{\mathcal{C}}(FG\ker \epsilon_Y, X)/\Ext^{t-1}_{\mathcal{C}}(FGY, X) .\end{eqnarray*}
Furthermore, if $\mathcal{C}$ has enough $E$-injectives, then for all $t\geq 1$, the functor $\Sha^t(-, X)$ vanishes if and only if
$\Ext^{t-1}_{\mathcal{C}}(-, X)$ is left $E$-exact.
\end{corollary}

\section{Main application: splitting morphisms of morphisms.}

In this section, we study the special case of $\Sha$ which occurs in the following way: we begin with an abelian category $\mathcal{A}$, and we consider
the category $\mathcal{A}^{(\bullet\rightarrow\bullet)}$ of arrows in $\mathcal{A}$. Clearly there is a forgetful functor
$\mathcal{A}^{(\bullet\rightarrow\bullet)}\rightarrow \mathcal{A}\times\mathcal{A}$ sending a morphism $f$ in $\mathcal{A}$ to the pair
$(\dom f,\cod f)$ consisting of the domain of $f$ and the codomain of $f$. For this section we will let $\Sha^n(f,g)$ be the kernel
of the map
\[ \Ext^n_{\mathcal{A}^{(\bullet\rightarrow\bullet)}}(f,g) \rightarrow \Ext^n_{\mathcal{A}}(\dom f,\dom g) \times \Ext^n_{\mathcal{A}}(\cod f,\cod g).\]
In other words, $\Sha^n(f,g)$ is the group of equivalence classes of diagrams
\[\xymatrix{ 0 \ar[r]\ar[d] & \dom g\ar[r]\ar[d]^g & X_1 \ar[r]\ar[d] & \dots \ar[r] & X_n \ar[r]\ar[d] & \dom f \ar[r]\ar[d]^f & 0\ar[d] \\
0 \ar[r] & \cod g\ar[r] & Y_1 \ar[r] & \dots \ar[r] & Y_n \ar[r] & \cod f \ar[r] & 0 }\]
in $\mathcal{A}$, where the top row and the bottom row are each {\em split} exact sequences. Such a diagram
represents zero in $\Sha^n(f,g)$ if and only if there exists splittings of the top and bottom rows which are {\em compatible} with
the vertical maps. 
This is an important special case of Definition~\ref{def of sha}.

\begin{prop}\label{adjunction prop}
Let $(\bullet\bullet)$ denote the category with two objects and no non-identity morphisms, and let $(\bullet\rightarrow\bullet)$ denote the category
with two objects and a single non-identity morphism from one object to the other. Then we have the abelian category $\mathcal{A}^{(\bullet\bullet)} \cong \mathcal{A}\times \mathcal{A}$ and the abelian category $\mathcal{A}^{(\bullet\rightarrow\bullet)}$ of morphisms in $\mathcal{A}$, and the exact faithful functor 
$G: \mathcal{A}^{(\bullet\rightarrow\bullet)} \rightarrow \mathcal{A}^{(\bullet\bullet)}$, that is, the functor induced by the inclusion of the subcategory
$(\bullet\bullet)\hookrightarrow (\bullet\rightarrow\bullet)$.
The functor $G$ has a resolving left adjoint $F: \mathcal{A}^{(\bullet\bullet)} \rightarrow \mathcal{A}^{(\bullet\rightarrow\bullet)}$
given by:
\[ F(X,Y) = \left( X\stackrel{\left[ \id_X\ 0 \right]}{\longrightarrow} X\oplus Y \right) .\]

Consequently, $FG(X \stackrel{f}{\longrightarrow} Y) = (X \stackrel{\left[ \id_X\ 0 \right]}{\longrightarrow} X\oplus Y)$, and
the counit map $\epsilon_f: FGf \rightarrow f$ of the adjunction $F\dashv G$ consists of the horizontal maps in the commutative diagram in $\mathcal{A}$:
\[\xymatrix{
X \ar[d]_{\left[ \id_X\ 0 \right]} \ar[r]^{\id_X} & X \ar[d]^f \\
X\oplus Y \ar[r]_{\left[f\ \id_Y\right]^{\perp}} & Y .
}\]
\end{prop}
\begin{proof} Elementary.\end{proof}

The following is now a corollary of Theorem~\ref{main general thm}:
\begin{corollary}\label{diagram splittings corollary}
Let $\mathcal{A}$ be an abelian category. Let $f: X^{\prime}\rightarrow X$ and $g: Y^{\prime}\rightarrow Y$ be morphisms in $\mathcal{A}$.
Then the following are equivalent:
\begin{itemize}
\item $\Sha^n(f, g) \cong 0$.
\item Each length $n+2$ diagram in $\mathcal{A}$ with exact rows
\begin{equation}\label{diag 1}\xymatrix{ 
0 \ar[r] \ar[d] & Y^{\prime}\ar[d]^g \ar[r] & E_1^{\prime} \ar[d]\ar[r] & \dots \ar[r] & E_n^{\prime} \ar[r]\ar[d] & X^{\prime}\ar[d]^f \ar[r] & 0\ar[d] \\
0 \ar[r]  & Y \ar[r] & E_1 \ar[r] & \dots \ar[r] & E_n \ar[r] & X \ar[r] & 0 }\end{equation}
is compatibly split if and only if its top row and its bottom row are both split.
\end{itemize}

If $n=1$, the above conditions are furthermore each equivalent to:
\begin{itemize}
\item For each pair of maps $i: Y^{\prime}\rightarrow Z$ and $h: Z^{\prime}\rightarrow Z$ in $\mathcal{A}$, there exist maps $j,k$ making the following
diagram commute:
\[ \xymatrix{ Y^{\prime}\ar[rd]_{i} \ar[r]^(0.4){\left[ -g\ \id_{Y^{\prime}}\right]} & Y\oplus Y^{\prime} \ar@{-->}[d]^j & Y^{\prime}\ar[l]_(0.4){\left[0\ \id_{Y^{\prime}}\right]} \ar@{-->}[d]^k \\
 & Z & Z^{\prime}\ar[l]^h. }\]
\end{itemize}
\end{corollary}
\begin{proof}
This is the case of Theorem~\ref{main general thm} in which $\mathcal{C} = \mathcal{A}^{(\bullet\rightarrow\bullet)}$; in which $I$ consists of only a single element,
which we shall write $I = \{i\}$; in which $\mathcal{C}_{i} = \mathcal{A}^{(\bullet\bullet)} \cong \mathcal{A}\times\mathcal{A}$;
and in which the functors $F_i,G_i$ are the functors $F,G$
of Proposition~\ref{adjunction prop}.
A length $n+2$ exact sequence in $\mathcal{C}$ is then precisely a diagram of the form~\ref{diag 1} with exact rows, and this diagram
is split (in $\mathcal{A}^{(\bullet\bullet)}$) after applying $G$ if and only if its top row and bottom row are each split (in $\mathcal{A}$).

The third condition of Theorem~\ref{main general thm} is equivalent to the third condition given above, using Proposition~\ref{adjunction prop}
to identify $\ker \epsilon_f$ and $FGf$.
\end{proof}

\begin{lemma}\label{identification of kernel of adjunction}
Let $\mathcal{A}$ be an abelian category and let $f: X \rightarrow Y$ be a morphism in $\mathcal{A}$. 
Then the kernel $\ker \epsilon_f$ of the counit of the adjunction $F\dashv G$ on $\mathcal{A}^{(\bullet\rightarrow\bullet)}$ is isomorphic to the map
$0 \rightarrow X$.
\end{lemma}
\begin{proof}
The short exact sequence in $\mathcal{A}^{(\bullet\rightarrow\bullet)}$ 
\begin{equation}\label{diag 3} 0 \rightarrow \ker \epsilon_f \rightarrow FGf\rightarrow f \rightarrow 0\end{equation}
is the commutative diagram with exact rows in $\mathcal{A}$
\begin{equation}\label{diag 4}\xymatrix{ 
0 \ar[r]\ar[d] & 0 \ar[d]\ar[r]  & X \ar[r]^{\id_X}\ar[d]^{\left[ \id_X\ 0\right]} & X \ar[d]^f\ar[r] & 0 \ar[d] \\
0 \ar[r] & \ker \left[ f\ \id_Y\right] \ar[r]  & X\oplus Y \ar[r]_(0.6){\left[ f\ \id_Y\right]^{\perp}} & Y \ar[r] & 0. }\end{equation}
All we need is to produce an isomorphism $\ker \left[ f\ \id_Y\right] \cong X$.
Consider the commutative diagram with exact rows in $\mathcal{A}$:
\begin{equation}\label{diag 2}\xymatrix{ 
0 \ar[r]\ar[d] & \ker \left[ f\ \id_Y\right] \ar[r]\ar[d]  & X\oplus Y\ar[d]^{m} \ar[r]^(0.6){\left[ f\ \id_Y\right]^{\perp}} & Y\ar[d]^{\id_Y} \ar[r] & 0 \\
0 \ar[r] & X \ar[r]  & X\oplus Y \ar[r]_(0.6){\left[ 0\ \id_Y\right]^{\perp}} & X \ar[r] & 0  }\end{equation}
where $m$ is given by the matrix of maps
\[ m = \left[ \begin{array}{ll} \id_X & f \\ 0 & \id_Y \end{array} \right] .\]
The map $m$ is invertible, with inverse given by
\[ m^{-1} = \left[ \begin{array}{ll} \id_X & -f \\ 0 & \id_Y \end{array}\right] ,\]
so the vertical maps $m$ and $\id_Y$ in diagram~\ref{diag 2} are isomorphisms. Hence the map $\ker \left[ f\ \id_Y\right] \rightarrow X$
of kernels is also an isomorphism.
\end{proof}

\begin{theorem}\label{main thm on morphisms}
Let $\mathcal{A}$ be an abelian category and let $E$ be the allowable class on the abelian category $\mathcal{A}^{(\bullet\rightarrow\bullet)}$
of morphisms in $\mathcal{A}$ consisting of all short exact sequences of the form
\[ 0 \rightarrow \ker \epsilon_f \rightarrow FGf \rightarrow f \rightarrow 0\]
for all objects $f$ in $\mathcal{A}^{(\bullet\rightarrow\bullet)}$, i.e., for all morphisms $f$ in $\mathcal{A}$.
(This is a special case of Definition~\ref{def of allowable class}.)
Here $F,G$ are as in Proposition~\ref{adjunction prop}.
Then each of the following statements is true:
\begin{itemize}
\item A morphism $f$ of $\mathcal{A}$ is $E$-projective if and only if $f$ is a split monomorphism.
\item The category $\mathcal{A}^{(\bullet\rightarrow\bullet)}$ has enough $E$-projectives, that is, every object in $\mathcal{A}^{(\bullet\rightarrow\bullet)}$ is the codomain of some $E$-epimorphism
with $E$-projective domain.
\item The allowable class $E$ is {\em hereditary}, i.e., the relative $\Ext$-groups $\Ext_{\mathcal{C}/E}^n(f, g)$ vanish for all $f,g$
if $n>1$. 
\item A morphism $f$ of $\mathcal{A}$ is $E$-injective if and only if $f$ is a split epimorphism.\end{itemize}
\end{theorem}
\begin{proof}
First, suppose $f: X \rightarrow Y$ is $E$-projective. Then consider the diagram in $\mathcal{A}^{(\bullet\rightarrow\bullet)}$:
\[ \xymatrix{ & f\ar@{-->}[ld]\ar[rd] & \\ g \ar[rr] & & h }\]
given by the following commutative diagram in $\mathcal{A}$:
\[ \xymatrix{  & X \ar@{-->}[ddl]\ar[d]^(0.6){f} \ar[ddr]^{\id_X} & \\
 & Y \ar@{-->}[ddl]\ar[ddr] & \\
X \ar[d]_{g=\id_X} \ar[rr]^{\id_X} & & X \ar[d]^{h=0} \\
X \ar[rr] & & 0.}\]
Since $f$ is assumed $E$-projective, maps exist as in the dotted lines which make the diagram commute.
The top dotted map must then be the identity map on $X$. This in turn forces the bottom dotted map to be a retraction of $f$. So $f$ must be split monic.

Now suppose $f: X \rightarrow Y$ is split monic. We want to show that it is $E$-projective. Since $f$ is split monic, it can be written as the direct sum 
$f = f^{\prime}\oplus f^{\prime\prime}$ with $f^{\prime}$ an isomorphism and $f^{\prime\prime}$ having zero domain. So it suffices to show that every isomorphism 
in $\mathcal{A}$ is $E$-projective and every map with zero domain in $\mathcal{A}$ is $E$-projective.
For the isomorphisms this is trivial. For the maps $f^{\prime\prime}$ with zero domain, we must simply produce a map $\ell$ to fill in the dotted line in each commutative
diagram of the form
\[\xymatrix{ 
 & 0\ar[d] \ar[rdd] \ar[ldd] & \\ 
 & Y \ar[rdd]_(0.7){j} \ar@{-->}[ldd]^(0.7){\ell}  & \\
M^{\prime} \ar[rr]^{\id_{M^{\prime}}} \ar[d]_{\left[ 0\ i\right]} & & M^{\prime}\ar[d]^i \\
M^{\prime}\oplus M \ar[rr]^{\left[ i\ \id_M\right]^{\perp}} & & M. }\]
The desired map $\ell$ is $\ell = \left[ 0\ \id_M\right]^{\perp}\circ j$.

So we have proven that a morphism in $\mathcal{A}$ is $E$-projective if and only if the morphism is split monic.
Now for any morphism $f: X \rightarrow Y$ in $\mathcal{A}$ we have the diagram
\[\xymatrix{
 0 \ar[r] \ar[d] & 0\ar[r] \ar[d] & X\ar[r]^{\id_X}\ar[d]^{\left[ \id_X\ 0\right]} & X\ar[d]^f\ar[r] & 0 \ar[d] \\
 0 \ar[r]  & X \ar[r]  & X\oplus Y\ar[r]_{\left[ f\ \id_Y\right]} & Y \ar[r] & 0 }\]
in which the lower-right copy of $X$ is (isomorphic to) $X$ by virtue of Lemma~\ref{identification of kernel of adjunction}.
So in the equivalent short exact sequence in $\mathcal{A}^{(\bullet\rightarrow\bullet)}$,
\[ 0 \rightarrow \ker \epsilon_f \rightarrow FGf\stackrel{\epsilon_f}{\longrightarrow} f \rightarrow 0,\]
both the maps $\ker \epsilon_f$ and $FGf$ in $\mathcal{A}$ are split monomorphisms, hence they are each 
$E$-projective objects of $\mathcal{A}^{(\bullet\rightarrow\bullet)}$. Hence every object of $\mathcal{A}^{(\bullet\rightarrow\bullet)}$
has a length $1$ $E$-resolution by $E$-projective objects. Hence $E$ is hereditary and $\mathcal{A}^{(\bullet\rightarrow\bullet)}$ has
enough $E$-projectives.

The proof that an object $f$ of $\mathcal{A}^{(\bullet\rightarrow\bullet)}$ is $E$-injective if and only if $f$ is split epic is dual to the proof that
$f$ is $E$-projective if and only if it is split monic, given above.
\end{proof}

\begin{corollary}\label{split monics cor}
Let $\mathcal{A}$ be an abelian category and let $E$ be as in Theorem~\ref{main thm on morphisms}.
Let $h: Z^{\prime}\rightarrow Z$ be a morphism in $\mathcal{A}$. Then the following are equivalent:
\begin{itemize}
\item 
If 
\begin{equation}\label{diag 2000}\xymatrix{ 0 \ar[r]\ar[d] & X^{\prime}\ar[d]^f \ar[r]^{i^{\prime}} & Y^{\prime}\ar[r]^{\pi^{\prime}}\ar[d]^g & Z^{\prime} \ar[r]\ar[d]^h & 0 \ar[d] \\
0 \ar[r] & X \ar[r]^i & Y \ar[r]^{\pi} & Z \ar[r] & 0 }\end{equation}
is a commutative diagram in $\mathcal{C}$ with split exact rows, then there exists a splitting $r^{\prime}: Y^{\prime}\rightarrow Y$ of $i^{\prime}$
and a splitting $r: Y\rightarrow X$ of $i$ such that $f\circ r^{\prime} = r\circ g$.
\item 
If diagram~\ref{diag 2000}
is a commutative diagram in $\mathcal{C}$ with split exact rows, then there exists a splitting $s^{\prime}: Y^{\prime}\rightarrow Y$ of $\pi^{\prime}$
and a splitting $s: Y\rightarrow X$ of $\pi$ such that $g\circ s^{\prime} = s\circ h$.
\item
The first compatible splitting obstruction group $\Sha^1(h,f)\cong 0$ is trivial for all morphisms $f$ in $\mathcal{A}$.
\item 
The morphism $h$ is $E$-projective in $\mathcal{A}^{(\bullet\rightarrow\bullet)}$.
\item 
The morphism $h$ is split monic.
\end{itemize}
\end{corollary}
\begin{proof}
Follows immediately from Theorem~\ref{main general thm} and Theorem~\ref{main thm on morphisms}.\end{proof}

\begin{corollary}\label{split epics cor}
Let $\mathcal{A}$ be an abelian category and let $E$ be as in Theorem~\ref{main thm on morphisms}.
Let $f: X^{\prime}\rightarrow X$ be a morphism in $\mathcal{A}$. Then the following are equivalent:
\begin{itemize}
\item 
If diagram~\ref{diag 2000}
is a commutative diagram in $\mathcal{C}$ with split exact rows, then there exists a splitting $r^{\prime}: Y^{\prime}\rightarrow Y$ of $i^{\prime}$
and a splitting $r: Y\rightarrow X$ of $i$ such that $f\circ r^{\prime} = r\circ g$.
\item 
If diagram~\ref{diag 2000}
is a commutative diagram in $\mathcal{C}$ with split exact rows, then there exists a splitting $s^{\prime}: Y^{\prime}\rightarrow Y$ of $\pi^{\prime}$
and a splitting $s: Y\rightarrow X$ of $\pi$ such that $g\circ s^{\prime} = s\circ h$.
\item
The first compatible splitting obstruction group $\Sha^1(h,f)\cong 0$ is trivial for all morphisms $h$ in $\mathcal{A}$.
\item 
The morphism $f$ is $E$-injective in $\mathcal{A}^{(\bullet\rightarrow\bullet)}$.
\item 
The morphism $f$ is split epic.
\end{itemize}
\end{corollary}
\begin{proof}
Follows immediately from Theorem~\ref{main general thm} and Theorem~\ref{main thm on morphisms}.\end{proof}

\begin{corollary} {\bf (Compatible splitting duality.)}\label{splitting duality}
Let $\mathcal{A}$ be an abelian category with enough projectives and let $E$ be as in Theorem~\ref{main thm on morphisms}.
Let $f: X \rightarrow Y$ and $g: V\rightarrow W$ be morphisms in $\mathcal{A}$. 
Then, for all $t\geq 1$, we have isomorphisms
\begin{eqnarray*} \Sha^t(f,g) & \cong & (R^1_E\Ext^{t-1}_{\mathcal{A}^{(\bullet\rightarrow\bullet)}}(-,g))(f) \\
 & \cong & \Ext^{t-1}_{\mathcal{A}^{(\bullet\rightarrow\bullet)}}(FG\ker\epsilon_f, g)/\Ext^{t-1}_{\mathcal{A}^{(\bullet\rightarrow\bullet)}}(FGf, g)\end{eqnarray*}
and an exact sequence
\[ 0 \rightarrow \Sha^t(f,g) \rightarrow \Ext^t_{\mathcal{A}^{(\bullet\rightarrow\bullet)}}(f,g) \rightarrow 
 \Ext^t_{\mathcal{A}}(X,V) \times\Ext^t_{\mathcal{A}}(Y,W) \rightarrow
 \Ext^t_{\mathcal{A}}(X,W) \rightarrow \Sha^{t+1}(f, g)\rightarrow 0.\]
\end{corollary}
\begin{proof}
The isomorphisms are corollaries of Theorem~\ref{main thm on morphisms} and Corollary~\ref{compatible splitting duality in hereditary case}.
The exact sequence follows from the given isomorphisms, together with the observation (in Lemma~\ref{identification of kernel of adjunction}) that
$\ker \epsilon_f$ is the map $0 \rightarrow X$, hence
$\Ext^t_{\mathcal{A}^{(\bullet\rightarrow\bullet)}}(\ker\epsilon_f,g) \cong \Ext^t_{\mathcal{A}}(X, W)$.
\end{proof}

\bibliography{/home/asalch/texmf/tex/salch}{}
\bibliographystyle{plain}
\end{document}